\documentclass{article}

\usepackage{amsfonts,amsmath, amsthm, amssymb}
\usepackage[latin1]{inputenc}
\usepackage[english]{babel}
\usepackage{graphicx}
\usepackage{enumerate}

\def\td{tree-decom\-po\-si\-tion}

\newcommand{\indsub}{\subseteq_{\mathrm i}}
\newcommand{\tw}{\rm tw}

\newcommand{\F}{\mathcal F}
\newcommand{\G}{\mathcal G}

\newcommand{\V}{\mathcal V}

\newcommand{\Q}{\mathcal Q}
\newcommand{\R}{\mathcal R}
\newcommand{\PP}{\mathcal P}

\newcommand{\sub}{\subseteq}

\newcommand{\comment}[1]{}

\newtheorem{theorem}{Theorem}
\newtheorem{lemma}[theorem]{Lemma}
\newtheorem{corollary}[theorem]{Corollary}

\theoremstyle{definition}

\def\?#1{\vadjust{\vbox to 0pt{\vss\vskip-8pt\leftline{%
     \llap{\hbox{\vbox{\pretolerance=-1
     \doublehyphendemerits=0\finalhyphendemerits=0
     \hsize16truemm\tolerance=10000\small
     \lineskip=0pt\lineskiplimit=0pt
     \rightskip=0pt plus16truemm\baselineskip8pt\noindent
     \hskip0pt        
     #1\endgraf}\hskip7truemm}}}\vss}}}
     
     \title{In absence of long chordless cycles, large tree-width becomes a local phenomenon}
     \author{Daniel Wei\ss{}auer}
     \date{}

\begin{document}
     	\maketitle
     
	\begin{abstract}
	We prove that, for all~$\ell$ and~$s$, every graph of sufficiently large tree-width contains either a complete bipartite graph~$K_{s,s}$ or a chordless cycle of length greater than~$\ell$.
\end{abstract}

     \begin{section}{Introduction}
     
%
     
     In an effort to make the statement in the title precise, let us call a graph parameter~$P$ \emph{global} if there is a constant~$c$ such that for all~$k$ and~$r$ there exists a graph~$G$ for which every subgraph~$H$ of order at most~$r$ satisfies $P(H) < c$, while $P(G) > k$. The intention here is that~$P$ being small, even bounded by a constant, on subgraphs of bounded order does not provide a bound on~$P(G)$.     
     
      Tree-width is a global parameter (we may take $c = 2$), as is the chromatic number (with $c = 3$). Indeed, it is a classic result of Erd\H{o}s~\cite{erdos59girth} that for all~$k$ and~$r$ there exists a graph of chromatic number~$>k$ for which every subgraph on at most~$r$ vertices is a forest.
      
          It is well-known (see~\cite{DiestelBook16noEE}) that the situation changes when we restrict ourselves to \emph{chordal graphs}, graphs without chordless cycles of length~$\geq 4$:
        \begin{equation} \label{tw chordal graphs}
    	\text{$\forall k: \,$ Every $K_{k+1}$-free chordal graph has tree-width~$<\! k$.}
     \end{equation}
     Hence the only obstruction for a chordal graph to have small tree-width is the presence of a large clique. Since the chromatic number of a graph is at most its tree-width plus one~(\cite{DiestelBook16noEE}), the same is true for the chromatic number. In particular, tree-width and chromatic number are \emph{local} parameters for the class of chordal graphs.

      In 1985, Gy\'{a}rf\'{a}s~\cite{gyarfas87conj} made a famous conjecture which implies that chromatic number is a local parameter\footnote{Indeed, in terms of our earlier definition, (\ref{long hole chi bound}) implies that given any integer~$c$, there exists a~$k$ such that every \mbox{$\ell$-chordal} graph of chromatic number~$>k$ has a subgraph of order~$ \leq c$ and chromatic number~$\geq c$.} for the larger class of \emph{\mbox{$\ell$-chordal}} graphs, those which have no chordless cycle of length~$> \ell$:
     \begin{equation} \label{long hole chi bound}
     \text{$\forall \ell, r \; \exists k: \,$ Every $K_r$-free \mbox{$\ell$-chordal} graph is $k$-colourable.}
     \end{equation}

		This conjecture remained unresolved for 30 years and was proved only recently by Chudnovsky, Scott and Seymour~\cite{longholechibound}. In view of~(\ref{tw chordal graphs}), it is tempting to think that an analogue of~(\ref{long hole chi bound}) might hold with tree-width in place of chromatic number. Complete bipartite graphs, however, are examples of triangle-free \mbox{4-chordal} graphs of large tree-width. Therefore a verbatim analogue of~(\ref{long hole chi bound}) is not possible and any graph whose presence we can hope to force by assuming \mbox{$\ell$-chordality} and large tree-width will be bipartite. 
%
%
%
		 
		 		 On the positive side, Bodlaender and Thilikos~\cite{bodlaender97chordal} showed that every \emph{star} can be forced as a subgraph in \mbox{$\ell$-chordal} graphs by assuming large tree-width (see Section~\ref{sec: proof}). However, since stars have tree-width~1, this does not establish locality of tree-width in the sense of our earlier definition. Our main result is that in fact \emph{any} bipartite graph can be forced as a subgraph:

%
          
	     \begin{theorem} \label{main theorem}
     	Let $\ell \geq 4$ be an integer and~$F$ a graph. Then~$F$ is bipartite if and only if there exists an integer~$k$ such that every \mbox{$\ell$-chordal} graph of tree-width~$\geq k$ contains~$F$ as a subgraph.
     \end{theorem}          
          
	This shows that tree-width is local for \mbox{$\ell$-chordal} graphs: Given any integer~$c$, there exists an integer~$k$ such that every \mbox{$\ell$-chordal} graph of tree-width~$\geq k$ has a subgraph isomorphic to~$K_{c,c}$, which has order~$2c$ and tree-width~$c$.
          
	Theorem~\ref{main theorem} also has an immediate application to an Erd\H{o}s-P\'{o}sa type problem. Kim and Kwon~\cite{kim18holes} showed that chordless cycles of length~$>3$ have the Erd\H{o}s-P\'{o}sa property:
	
	\begin{theorem}[\cite{kim18holes}] \label{kim kwon ep}
		For every integer~$k$ there exists an integer~$m$ such that every graph~$G$ either contains~$k$ vertex-disjoint chordless cycles of length~$>3$ or a set~$X$ of at most~$m$ vertices such that $G - X$ is chordal.
	\end{theorem}
	
	They also constructed, for every integer $\ell \geq 4$, a family of graphs showing that the analogue of Theorem~\ref{kim kwon ep} for chordless cycles of length~$>\ell$ fails. We complement their negative result by proving that the Erd\H{o}s-P\'{o}sa property \emph{does} hold when restricting the host graphs to graphs not containing~$K_{s,s}$ as a subgraph.

     \begin{corollary} \label{erdos posa long holes}
     	
     	For all~$\ell, s$ and~$k$ there exists an integer~$m$ such that every \mbox{$K_{s,s}$-free} graph~$G$ either contains~$k$ vertex-disjoint chordless cycles of length~$>\ell$ or a set~$X$ of at most~$m$ vertices such that $G - X$ is \mbox{$\ell$-chordal}.
     \end{corollary}
     
     The paper is organised as follows. Section~\ref{sec: prelims} contains some basic definitions. Theorem~\ref{main theorem}, our main result, is proved in Section~\ref{sec: proof}. In Section~\ref{sec: erdos posa} we formally introduce the Erd\H{o}s-P\'{o}sa property, restate Corollary~\ref{erdos posa long holes} in that language and give a proof thereof. Section~\ref{sec: problems} closes with some open problems.

	\end{section}

	\begin{section}{Notation and definitions} \label{sec: prelims}
			All graphs considered here are finite and undirected and contain neither loops nor parallel edges. Our notation and terminology mostly follow that of~\cite{DiestelBook16noEE}.
			
	For two graphs~$G$ and~$H$, we say that~$G$ is \emph{$H$-free} if~$G$ does not contain a subgraph isomorphic to~$H$. 
		Given a tree~$T$ and $s, t \in T$, we write $sTt$ for the unique $s$-$t$-path in~$T$. Given a graph~$G$ and a set~$X$ of vertices of~$G$, a path $P \sub G$ is an \emph{$X$-path} if it contains at least one edge and meets~$X$ precisely in its endvertices. A \emph{separation} of~$G$ is a tuple $(A, B)$ with $V = A \cup B$ such that there are no edges between $A \setminus B$ and $B \setminus A$. The \emph{order} of $(A, B)$ is the number of vertices in $A \cap B$. We call the separation $(A,B)$ \emph{tight} if for all $x, y \in A \cap B$, both~$G[A]$ and~$G[B]$ contain an $x$-$y$-path with no internal vertices in $A \cap B$.
		
		Given an integer~$k$, a set~$X$ of at least~$k$ vertices of~$G$ is a \emph{$k$-block} if it is inclusion-maximal with the property that for every separation $(A,B)$ of \mbox{order~$<\! k$}, either $X \sub A$ or $X \sub B$. By Menger's Theorem, $G$ then contains~$k$ internally disjoint paths between any two non-adjacent vertices in~$X$.
			
			A \emph{\td\ of~$G$} is a pair $(T, \V)$, where~$T$ is a tree and $\V = (V_t)_{t \in T}$ a family of sets of vertices of~$G$ such that for every $v \in V(G)$, the set of $t \in T$ with $v \in V_t$ induces a non-empty subtree of~$T$ and for every edge $vw \in E(G)$ there is a $t \in T$ with $v, w \in V_t$. If $(T, \V)$ is a \td\ of~$G$, then every $st \in E(T)$ induces a separation $(G_{s}^t, G_{t}^s)$ of~$G$, where~$G_{x}^y$ is the union of~$V_u$ for all $u \in T$ for which $y \notin uTx$. Note that $G_s^t \cap G_t^s = V_s \cap V_t$. We call $(T, \V)$ \emph{tight} if every separation induced by an edge of~$T$ is tight.
			
			 Given $t \in T$, the \emph{torso at~$t$} is the graph obtained from $G[V_t]$ 	by adding, for every neighbor~$s$ of~$t$, an edge between any two non-adjacent vertices in~$ V_s \cap V_t$.
			
			Given graphs~$G$ and~$H$, a \emph{subdivision of~$H$ in~$G$} consists of an injective map $\eta: V(H) \to V(G)$ and a map~$P$ which assigns to every edge $xy \in E(H)$ an \mbox{$\eta(x)$-$\eta(y)$-path} $P^{xy} \sub G$ so that the paths $(P^{xy} \colon xy \in E(H) )$ are internally disjoint and no~$P^{xy}$ has an internal vertex in $X := \eta(V(H))$. The vertices in~$X$ are called \emph{branchvertices}. For an integer~$r$, the subdivision is a \emph{$(\leq r)$-subdivision} if every path~$P^{xy}$ has length at most~$r$. When~$H$ is a complete graph, the map~$\eta$ is irrelevant and we only keep track of the set~$X$ of branchvertices and the family $( P^{xy} \colon x, y \in X )$.

	\end{section}

	\begin{section}{Proof of Theorem~\ref{main theorem}} \label{sec: proof}

	As observed in the introduction, the complete bipartite graphs~$K_{s,s}$ show that no bound on the tree-width of \mbox{$F$-free} \mbox{$\ell$-chordal} graphs exists if~$F$ is not bipartite. We now prove that~$F$ being bipartite is sufficient. Since every bipartite graph is a subgraph of some~$K_{s,s}$, it suffices to prove Theorem~\ref{main theorem} for the case $F = K_{s,s}$. 
	
	Our proof is a cascade with three steps. First, we show that sufficiently large tree-width forces the presence of a $k$-block.

		\begin{lemma} \label{tw to block}
			Let~$ \ell, k$ and $t \geq 2(\ell-2)(k-1)^2$ be positive integers. Then every \mbox{$\ell$-chordal} graph of tree-width~$\geq t $ contains a $k$-block.
		\end{lemma}
		
	We then prove that the existence of a $k$-block yields a bounded-length subdivision of a complete graph.
		
		
			\begin{lemma}		\label{block to subdiv}
		Let $ \ell, m$ and $k \geq 5m^2 \ell/4$ be positive integers. Then every \mbox{$\ell$-chordal} graph that contains a $k$-block contains a $(\leq 2 \ell -3)$-subdivision of~$K_m$.
	\end{lemma}

	In the last step, we show that such a bounded-length subdivision gives rise to a copy of~$K_{s,s}$.
		
		
			\begin{lemma} \label{subdiv to compbip alt}
	For all integers~$\ell$ and~$s$ there exists a $q >0 $ such that the following holds. Let $m, r$ be positive integers with $m \geq qr$. Then every \mbox{$\ell$-chordal} graph that contains a $(\leq r)$-subdivision of~$K_m$ contains~$K_{s,s}$ as a subgraph.
		\end{lemma}
		
		It is immediate that Theorem~\ref{main theorem} follows once we have established these three lemmas.
	
	\begin{subsection}{Proof of Lemma~\ref{tw to block}}
	
	A trivial obstacle to our search for a copy of~$K_{s,s}$ is the absence of vertices of high degree. Bodlaender and Thilikos~\cite{bodlaender97chordal} showed, however, that \mbox{$\ell$-chordal} graphs of bounded degree have bounded tree-width. Their exponential bound was later improved by Kosowski, Li, Nisse and Suchan~\cite{kosowski12chordal} and by Seymour~\cite{seymour16treechi}.
	
		\begin{theorem}[\cite{seymour16treechi}] \label{chordal maxdeg}
			Let~$\ell$ and~$\Delta$ be positive integers and~$G$ a graph. If~$G$ is \mbox{$\ell$-chordal} and has no vertices of degree greater than~$\Delta$, then the tree-width of~$G$ is at most $(\ell - 2)(\Delta - 1) + 1$.
		\end{theorem}
		
		By demanding large tree-width, we can therefore guarantee a large number of vertices of high degree. We now show that these are not all just scattered about the graph. It was shown by the author in~\cite{weissauer17block} that either there is a $k$-block or there is a tree-decomposition which separates the set of vertices of high degree into small pieces. This also follows, without explicit bounds, from a far more general result of Dvo\v{r}\'{a}k~\cite{dvorak12stronger}.
		
		\begin{theorem}[\cite{weissauer17block}] \label{structure blocks}
			Let $k \geq 3$ be a positive integer and~$G$ a graph. If~$G$ has no $k$-block, then there is a tight \td\ $(T, \V)$ of~$G$ such that every torso has fewer than~$k$ vertices of degree at least~$2(k-1)(k-2)$.
		\end{theorem}
		
	In fact, tightness of the \td\ is not explicit in~\cite[Theorem~1]{weissauer17block}, but is established in the proof as \emph{Lemma~6}.

	Now let $\ell, k$ and $t \geq 2( \ell - 2)(k-1)^2$ be positive integers. Let~$G$ be an \mbox{$\ell$-chordal} graph with no $k$-block. For $k = 2$, this means that~$G$ is acyclic and therefore has tree-width~1. Suppose from now on that $k \geq 3$. We show that the tree-width of~$G$ is less than~$t$.
			
			By Theorem~\ref{structure blocks}, there is a tight \td\ $(T, \V)$ of~$G$ such that every torso has fewer than~$k$ vertices of degree at least $d := 2(k-1)(k-2)$. Let $t \in T$ arbitrary, let~$N$ be the set of neighbors of~$t$ in~$T$ and let~$H$ be the torso at~$t$. We claim that~$H$ is \mbox{$\ell$-chordal}.
			
			Let $C \sub H$ be a chordless cycle. For every edge $xy \in E(C) \setminus E(G)$, there is some $s \in N$ with $x, y \in V_s \cap V_t$. Since $(T, \V)$ is tight, there exists an \mbox{$x$-$y$-path}~$P^{xy}$ in $G_{s}^t$ which meets~$V_t$ only in its endpoints. Observe that for every $s \in N$, $C$ contains at most two vertices of~$V_s$ and these are adjacent in~$C$. Hence we can replace every edge $xy \in E(C) \setminus E(G)$ by~$P^{xy}$ and obtain a chordless cycle~$C'$ of~$G$ with $|C'| \geq |C|$. Since~$G$ is \mbox{$\ell$-chordal}, it follows that $|C| \leq \ell$. This proves our claim.
			
			Now, let $A \sub V(H)$ be the set of all vertices of degree $\geq d$ in~$H$. Then $H - A$ is \mbox{$\ell$-chordal} and has no vertices of degree~$> d-1 $. By Theorem~\ref{chordal maxdeg}, the tree-width of $H - A$ is at most~$(\ell-2)(d-2) + 1 $. Therefore
			\[
			\tw(H) \leq |A| + \tw( H - A) \leq k + (\ell - 2)(d-2) < t.
			\]
			We have shown that every torso has tree-width~$<\! t$. We can then take a \td\ of width~$<\! t$ of each torso and combine all these to a tree-decomposition of width~$<\! t$ of~$G$. 
			\qed
	
	\end{subsection}	
	
	\begin{subsection}{Proof of Lemma~\ref{block to subdiv}}

		In general, the presence of a $k$-block does not guarantee the existence of any subdivision of~$K_m$ for $m \geq 5$. For example, take a rectangular $k^2 \times k$-grid, add $2(k+1)$ new vertices to the outer face and make each of these adjacent to~$k$ consecutive vertices on the perimeter of the grid (see Figure~\ref{planar block}). These new vertices are then a $k$-block in the resulting planar graph.
		
		\begin{figure}[htb] \label{planar block}
			\begin{center}
				\includegraphics[width=5cm]{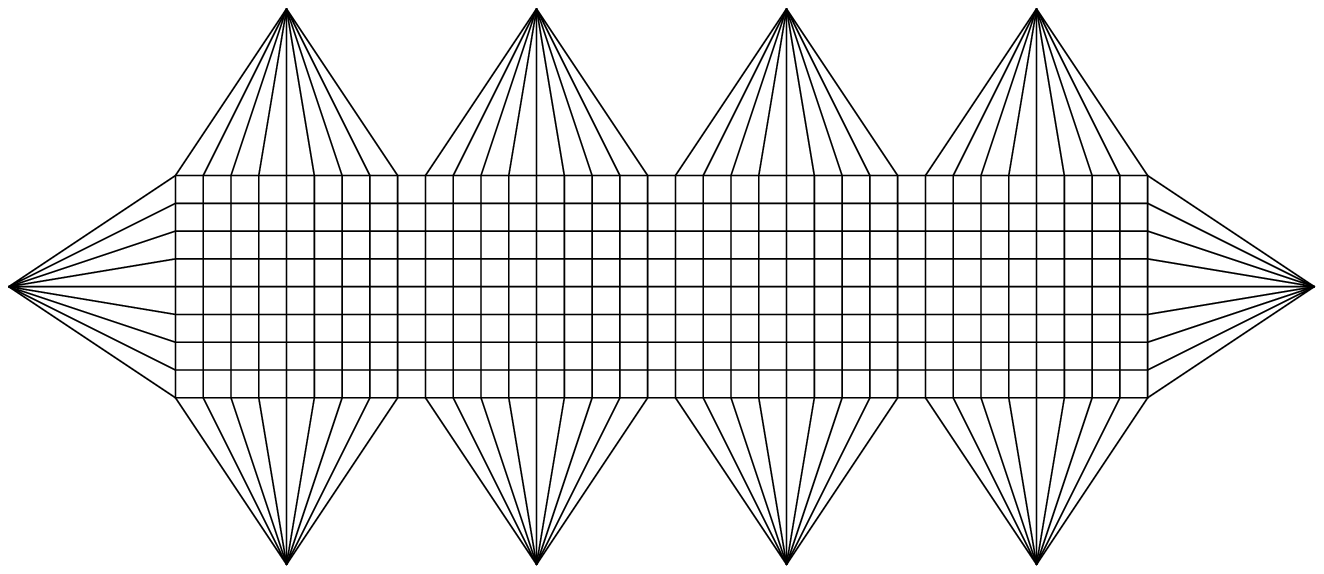}
			\end{center}
			\caption{A planar graph with a 9-block}
		\end{figure}
		
		Our aim in this section is to show that for \mbox{$\ell$-chordal} graphs, sufficiently large blocks do indeed yield bounded-length subdivisions of complete graphs.

	Let $\ell, m$ and $k \geq 5 m^2 \ell/4$ be positive integers. Let~$G$ be an \mbox{$\ell$-chordal} graph and $X \sub V(G)$ a $k$-block of~$G$. Let $L := 2 \ell - 3$. Assume for a contradiction that~$G$ contained no $(\leq L)$-subdivision of~$K_m$. Let $x, y \in X$ non-adjacent. Then~$G$ contains a set~$\PP^{xy}$ of~$k$ internally disjoint $x$-$y$-paths. Taking subpaths, if necessary, we may assume that each path in~$\PP^{xy}$ is induced. Let $p_0 := m + m^2(\ell - 2)$. \\
			
			\emph{Claim:} Fewer than~$p_0$ paths in~$\PP^{xy}$ have length~$> \ell/2$.
			\begin{proof}[Proof of Claim]
					Let~$\PP_0$ be the set of all paths in~$\PP^{xy}$ of length~$> \ell/2$ and $p := |\PP_0|$. Assume for a contradiction that $p \geq p_0$. Let $P, Q \in \PP_0$. Then $P \cup Q$ is a cycle of length~$> \ell$. Since~$G$ is \mbox{$\ell$-chordal}, $P \cup Q$ has a chord. This chord must join an internal vertex of~$P$ to an internal vertex of~$Q$. Choose such vertices $v^Q_P \in P$ and $v^P_Q \in Q$ so that the cycle $D := xPv^Q_Pv^P_QQx$ has minimum length. Note that~$D$ is an induced cycle and therefore has length at most~$\ell$. In particular, the segment of~$P$ joining~$x$ to~$v^Q_P$ has length at most~$\ell - 2$ and similarly for~$Q$ and~$v^P_Q$.
			
			For $P \in \PP_0$, let~$P'$ be a minimal subpath of~$P$ containing every vertex~$v_P^Q$, $Q \in \PP_0 \setminus \{ P \}$. Then $\PP := \{ P' \colon P \in \PP_0 \}$ is a family of~$p$ disjoint paths, each of length at most~$\ell - 3$, and~$G$ contains an edge between any two of them. Fix an arbitrary $\Q \sub \PP$ with $|\Q| = m$. Since $p \geq p_0$, every $Q \in \Q$ contains a vertex~$u_Q$ which has neighbors on at least~$m^2$ different paths in $\PP \setminus \Q$.
			
			Let $U := \{ u_Q \colon Q \in \Q \}$. We iteratively construct a $(\leq L)$-subdivision of~$K_m$ with branchvertices in~$U$. Let $t := \binom{m}{2}$ and enumerate the pairs of vertices of~$U$ arbitrarily as $e_1, \ldots, e_t$. In the $j$-th step, we assume that we have constructed a family $\R^j = (R_i)_{i < j}$ of internally disjoint $U$-paths of length at most~$L$, so that~$R_i$ joins the vertices of~$e_i$ and meets at most two paths in $\PP \setminus \Q$. We now find a suitable path~$R_j$.
			
			Let $Q^1, Q^2 \in \Q$ with $e_j = u_{Q^1}u_{Q^2}$. At most $2(j-1) < m^2$ paths in~$\PP \setminus \Q$ meet any of the paths in~$\R^j$. Since~$u_{Q^1}$ is adjacent to vertices on at least~$m^2$ different paths in~$\PP \setminus \Q$, there is a $P^1 \in \PP \setminus \Q$ which is disjoint from every~$R_i$, $i < j$, and contains a neighbor of~$u_{Q^1}$. We similarly  find a path $P^2 \in \PP \setminus \Q$ for~$u_{Q^{2}}$.  Since either $P^1 = P^2$ or~$G$ has an edge between~$P^1$ and~$P^2$, $P^1 \cup P^2 \cup \{ u_{Q^1}, u_{Q^2} \}$ induces a connected subgraph of~$G$ and therefore contains a $u_{Q^1}$-$u_{Q^2}$-path~$R_j$ of length at most~$L$, which meets only two paths in $\PP \setminus \Q$.
			
			Proceeding like this, we find the desired subdivision of~$K_m$ after~$t$ steps. This contradiction finishes the proof of the claim.
			\end{proof}
			
		Let $Y \sub X$ with $|Y| = m$. For any two non-adjacent $x, y \in Y$, let $\Q^{xy} \sub \PP^{xy}$ be the set of all $P \in \PP^{xy}$ of length at most~$\ell / 2$ which have no internal vertices in~$Y$. By the claim above, we have
	\[
	| \Q^{xy}| > k - p_0 - (m - 2) \geq \binom{m}{2} \frac{ \ell}{2} .
	\]	
	Pick one path $P \in \Q^{xy}$ for each pair of non-adjacent vertices $x, y \in Y$ in turn, disjoint from all previously chosen paths. Since $|Q^{xy}| \geq \binom{m}{2} \frac{ \ell}{2}$ and each path only has at most $\ell/2 - 1$ internal vertices which future paths need to avoid, we can always find a suitable such path~$P$. Together with all edges between adjacent vertices of~$Y$, this yields a $(\leq \ell/2)$-subdivision of~$K_m$ in~$G$ with branchvertices in~$Y$.
		\qed			\\
	
	We would like to point out that a modification of the above argument can be used to produce a $(\leq \ell / 2)$-subdivision of~$K_m$ if~$k$ is significantly larger. 
	
	Indeed, suppose we find a family~$\PP$ of~$p$ disjoint paths, each of length at most~$\ell - 3$, such that~$G$ contains an edge between any two of them. Then the subgraph~$H$ induced by $\bigcup_{P \in \PP} V(P)$ has at most $(\ell - 2)p$ vertices and at least~$\binom{p}{2}$ edges. One can then use a classic result of K\"{o}vari, S\'{o}s and Tur\'{a}n~\cite{kovarisosturan} to show that~$H$ contains a copy of~$K_{m,m^2}$ if~$p$ is sufficiently large. Since~$K_{m,m^2}$ contains a $(\leq 2)$-subdivision of~$K_m$, this establishes an upper bound on the number of paths of length~$>\ell/2$ in any~$\PP^{xy}$. The rest of the proof remains the same.

	\end{subsection}
	
	\begin{subsection}{Proof of Lemma~\ref{subdiv to compbip alt}}
	
	The combination of Lemma~\ref{tw to block} and Lemma~\ref{block to subdiv} already establishes that tree-width is a local parameter for \mbox{$\ell$-chordal} graphs. The purpose of Lemma~\ref{subdiv to compbip alt} is merely to narrow the set of bounded-order obstructions down as far as possible.	We will use the following theorem of K\"{u}hn and Osthus~\cite{kuhn04induced}.
	
		\begin{theorem}[\cite{kuhn04induced}] \label{kuhn osthus}
			For every integer~$s$ and every graph~$H$ there exists a~$d$ so that every graph with average degree at least~$d$ either contains~$K_{s,s}$ as a subgraph or contains an induced subdivision of~$H$.  
		\end{theorem}
		
		In fact, we only need the special case $H = C_{\ell + 1}$. This special case has a simpler proof which can be found in K\"{u}hn's PhD-thesis~\cite{kuhnthesis}. Fix an integer~$d$ so that every \mbox{$\ell$-chordal} graph of average degree at least~$d$ contains~$K_{s,s}$ as a subgraph. We prove the assertion of Lemma~\ref{subdiv to compbip alt} with $q := d^2 \frac{\ell^\ell}{4(\ell - 3)!}$.
		
	Let~$m, r$ be positive integers with $m \geq q r$ and let~$G$ be an \mbox{$\ell$-chordal} graph containing a $(\leq r$)-subdivision of~$K_m$. Let~$X$ be the set of branchvertices and $(P^{xy} \colon x, y \in X )$ the family of paths of the subdivision. Taking subpaths, if necessary, we may assume that every path is induced.
		
		Assume for a contradiction that~$G$ contained no copy of~$K_{s,s}$. By Theorem~\ref{kuhn osthus}, every subgraph of~$G$ contains a vertex of degree~$< \! d$. In particular, there is an independent set $Y \sub X$ with $|Y| \geq m/d$. Let~$H$ be the subgraph of~$G$ induced by $\bigcup_{x, y \in Y} V(P^{xy})$. Note that $|H| \leq r \binom{|Y|}{2}$. 
		
		Call an edge of~$H$ \emph{red} if it joins a vertex $x \in Y$ to an internal vertex of a path $P^{yz}$ with $x \notin \{ y, z \}$. Call an edge of~$H$ \emph{blue} if it joins an internal vertex of a path~$P^{wx}$ to an internal vertex of a path~$P^{yz}$ with $\{ w, x \} \neq \{ y, z \}$. We will show that~$H$ must contain many edges which are either red or blue, so that the average degree of~$H$ is at least~$d$.

		Fix an arbitrary cycle~$R$ with $V(R) = Y$. For any $Z \sub Y$ with $|Z| = \ell$, obtain the cycle~$R_Z$ with $V(R_Z) = Z$ by contracting every $Z$-path of~$R$ to a single edge. We then get a cycle $C_Z \sub H$ by replacing every edge $xy \in R_Z$ with the path~$P^{xy}$. Since each path~$P^{xy}$ has length at least~2 and~$H$ is \mbox{$\ell$-chordal}, the cycle~$C_Z$ must have a chord. Since~$Y$ is independent and every path~$P^{xy}$ is induced, the chord must be a red or blue edge of~$H$.
		
		Consider a red edge $xv \in E(H)$ with $x \in Y$, $v \in P^{yz}$ and $x \notin \{ y, z \}$. If this edge is a chord for a cycle~$C_Z$, then $\{ x, y, z \} \sub Z$. Hence it can only occur as a chord for at most 
		\[
		\binom{|Y|-3}{\ell -3} \leq \frac{ |Y|^{\ell - 3}}{(\ell - 3)!} 
		\] 
		choices of~$Z$. Similarly, every blue edge $uv \in E(H)$ with $u \in P^{wx}$, $v \in P^{yz}$ and $\{ w, x \} \neq \{ y, z \}$ can only be a chord of~$C_Z$ if $\{ w,x,y,z \} \sub Z$. This also happens for at most 
		\[
		\binom{|Y|-3}{\ell -3} \leq \frac{ |Y|^{\ell - 3}}{(\ell - 3)!} 
		\]
		 choices of~$Z$. Let~$f$ be the number of edges of~$H$ which are either red or blue. Since every $Z \sub Y$ with $|Z| = \ell$ gives rise to a chord, it follows that
		 \[
	\frac{|Y|^{\ell}}{\ell^{\ell}}	\leq  \binom{|Y|}{\ell} \leq f \frac{|Y|^{\ell - 3}}{(\ell - 3)!}  .
		 \]
		This shows that the average degree of~$H$ is
		\[
		d(H) \geq \frac{2f}{|H|} \geq \frac{4 (\ell - 3)!}{r \ell^{\ell}} |Y| \geq d .
		\]
		By Theorem~\ref{kuhn osthus}, $H$ contains a copy of~$K_{s,s}$.
			\qed

	\end{subsection}
	
	\end{section}

	\begin{section}{Erd\H{o}s-P\'{o}sa for long chordless cycles} \label{sec: erdos posa}
	
	A classic theorem of Erd\H{o}s and P\'{o}sa~\cite{erdosposa} asserts that for every integer~$k$ there is an integer~$r$ such that every graph either contains~$k$ disjoint cycles or a set of at most~$r$ vertices meeting every cycle. This result has been the starting point for an extensive line of research, see the survey by Raymond and Thilikos~\cite{raymond17survey}.

	Let~$\F, \G$ be classes of graphs and~$\leq$ a containment relation between graphs. We say that~$\F$ \emph{has the Erd\H{o}s-P\'{o}sa property for~$\G$ with respect to~$\leq$} if there exists a function~$f$ such that for every $G \in \G$ and every integer~$k$, either there are disjoint $Z_1, \ldots, Z_k \sub V(G)$ such that for every $1 \leq i \leq k$ there is an $F_i \in \F$ with $F_i \leq G[Z_i]$, or there is a $X \sub V(G)$ with $|X| \leq f(k)$ such that $F \not \leq G - X$ for every $F \in \F$. When~$\G$ is the class of all graphs, we simply say that~$\F$ \emph{has the Erd\H{o}s-P\'{o}sa property with respect to~$\leq$}. We write $F \sub G$ if~$F$ is isomorphic to a subgraph of~$G$ and $F \indsub G$ if~$F$ is isomorphic to an \emph{induced} subgraph of~$G$.
	
	The theorem of Erd\H{o}s and P\'{o}sa then asserts that the class of cycles has the Erd\H{o}s-P\'{o}sa property with respect to~$\sub$. This implies that cycles also have the Erd\H{o}s-P\'{o}sa property with respect to~$\indsub$. It is known that for every~$\ell$, the class of cycles of length~$> \ell$ has the Erd\H{o}s-P\'{o}sa property with respect to~$\sub$, see~\cite{thomassen88posa, birmele07posa, mousset17posa}. Recently, Kim and Kwon~\cite{kim18holes} proved that cycles of length~$>3$ possess the Erd\H{o}s-P\'{o}sa property with respect to~$\indsub$:
	
	\begin{theorem}[\cite{kim18holes}] \label{kim kwon ep precise}
		There exists a constant~$c$ such that for every integer~$k$, every graph~$G$ either contains~$k$ vertex-disjoint chordless cycles of length~$>3$ or a set~$X$ of at most $ck^2 \log k$ vertices such that $G - X$ is chordal.
	\end{theorem}
	
	 In contrast, Kim and Kwon~\cite{kim18holes} showed that, for any given $\ell \geq 4$, cycles of length~$> \ell$ do \emph{not} have the Erd\H{o}s-P\'{o}sa property with respect to~$\indsub$. For any given~$n$, they constructed a graph~$G_n$ with no two disjoint chordless cycles of length~$> \ell$, for which no set of fewer than~$n$ vertices meets every chordless cycle of length~$>\ell$ in~$G_n$. This graph~$G_n$ contains a copy of~$K_{n,n}$. We show that this is essentially necessary:

	
	\begin{corollary} \label{erdos posa long holes excl}
		For all integers~$\ell$ and~$s$, the class of cycles of length~$> \ell$ has the Erd\H{o}s-P\'{o}sa property for the class of $K_{s,s}$-free graphs with respect to~$\indsub$.
	\end{corollary}

	This follows from Theorem~\ref{main theorem} by a standard argument. Since the proof is quite short, we provide it for the sake of completeness. First, recall the following consequence of the Grid Minor Theorem of Robertson and Seymour~\cite{GMV}.
	
	\begin{theorem}[\cite{GMV}] \label{tree-width partition}
		For all positive integers~$p$ and~$q$ there exists an~$r$ such that for every graph~$G$ with tree-width~$\geq r$, there are disjoint $Z_1, \ldots, Z_p \sub V(G)$ such that~$G[Z_i]$ has tree-width~$\geq q$ for every $1 \leq i \leq p$.
	\end{theorem}
	
	\begin{proof}[Proof of Corollary~\ref{erdos posa long holes excl}]
		Let~$k$ be an integer. By Theorem~\ref{main theorem} there exists an integer~$t$ such that every \mbox{$\ell$-chordal} graph with tree-width~$\geq t$ contains~$K_{s,s}$. By Theorem~\ref{tree-width partition}, there exists an~$r$ such that every graph with tree-width~$> r$ has~$k$ vertex-disjoint subgraphs of tree-width~$\geq t$.

		Let~$G$ be a $K_{s,s}$-free graph. We show that either~$G$ contains~$k$ disjoint chordless cycles of length~$>\ell$ or there is a set of at most $r(k-1)$ vertices whose deletion leaves an \mbox{$\ell$-chordal} graph.
		
		Suppose first that the tree-width of~$G$ was greater than~$r$. Let $Z_1, \ldots, Z_k$ be disjoint sets of vertices such that $G[Z_i]$ has tree-width~$\geq t$ for every~$i$. Then, by Theorem~\ref{main theorem}, every~$G[Z_i]$ must contain a chordless cycle of length~$> \ell$, since $K_{s,s} \not \sub G[Z_i]$. Therefore~$G$ contains~$k$ disjoint chordless cycles of length~$> \ell$.
		
		Suppose now that~$G$ had a \td\ $(T, \V)$ of width~$< \! r$. For every chordless cycle $C \sub G$ of length~$> \ell$, let $T_C \sub T$ be the subtree of all $t \in T$ with $V_t \cap V(C) \neq \emptyset$. If there are~$k$ disjoint such subtrees $T_{C^1}, \ldots, T_{C^k}$, then $C^1, \ldots, C^k$ are also disjoint and we are done. Otherwise, there exists $S \sub V(T)$ with $|S| < k$ which meets every subtree~$T_C$. Then $Z := \bigcup_{s \in S} V_s$ meets every chordless cycle of length~$> \ell$ in~$G$ and $|Z| \leq r(k-1)$.
		
	\end{proof}

	\end{section}

	\begin{section}{Open problems} \label{sec: problems}
	
	A large amount of research is dedicated to the study of \emph{$\chi$-boundedness} of graph classes, introduced by Gy\'{a}rf\'{a}s~\cite{gyarfas87conj}. Here, a class~$\G$ of graphs is called \emph{$\chi$-bounded} if there exists a function~$f$ so that for every integer~$k$ and $G \in \G$, either~$G$ contains a clique on~$k+1$ vertices or~$G$ is $f(k)$-colourable. This is a strengthening of the statement that chromatic number is a local parameter for~$\G$, with cliques being the only bounded-order subgraphs to look for.
	
	As we have seen, cliques are not the only reasonable local obstruction to having small tree-width. Nontheless, we may still ask 
	\begin{enumerate}
		\item For which classes of graphs is tree-width a local parameter?
		\item What kind of bounded-order subgraphs can we force on these classes?
		\item For which classes can we force large cliques by assuming large tree-width?
	\end{enumerate}
	
	We have seen in Section~\ref{sec: erdos posa} that long chordless cycles have the Erd\H{o}s-P\'{o}sa property for the class of $K_{s,s}$-free graphs. For which other classes is this true? Kim and Kwon~\cite{kim18holes} raised this question for the class of graphs without chordless cycles of length four.

	\end{section}

\bibliographystyle{plain}
\bibliography{collective}
     	
     \end{document}